\documentclass[12pt]{article}
\usepackage{latexsym,color,amsmath,amsthm,amssymb,amscd,amsfonts}
\usepackage{tikz}
\usetikzlibrary{arrows}
\usepackage{scalefnt}
\usepackage[font=small,labelfont=bf]{caption}
\usepackage{float}
\usepackage{graphicx} 
\usepackage{dsfont}
\usepackage{amsopn}
\usepackage{relsize}
\usepackage{mathrsfs}
\usepackage{upgreek}

\newtheoremstyle{nonum}{}{}{\itshape}{}{\bfseries}{.}{ }{\thmnote{#3}}

\newtheorem{thm}{Theorem}[section]
\newtheorem*{thm*}{Theorem}

\newtheorem{lem}[thm]{Lemma}

\newtheorem{rem}[thm]{Remark}

\newtheorem*{definition*}{Definition}

\newtheorem*{rems*}{Remarks}

\theoremstyle{nonum}

\newcommand{\R}{\mathbb R}

\newcommand{\iprod}[2]{\langle #1,#2 \rangle} 

\newcommand{\kn}{\mathcal{K}^n}   

\def\grad{{\nabla}}

\begin{document}
\title{Stability and Rate of Convergence of\\ the Steiner Symmetrization}
\date{}
\author{D.I. Florentin, A. Segal}
\maketitle
\begin{abstract}
We present a direct analytic method towards an estimate for the rate
of convergence (to the Euclidean Ball) of Steiner symmetrizations. To
this end we present a modified version of a known stability property
of the Steiner symmetrization.
\end{abstract}

\section{Introduction and results}\label{Sec_Intro}
Let $(\R^n, \iprod{\cdot}{\cdot})$ be some fixed Euclidean structure,
and let $\kn$ be the class of all compact convex sets in $\R^n$.
Denote by $D_n$ the Euclidean unit ball, by $S^{n-1}$ its boundary
and by $\kappa_n=|D_n|$ its Lebesgue measure. Fix a direction $u \in
S^{n-1}$ and denote its orthogonal hyperplane by $H=\{x\in \R^n:
\iprod{x}{u}=0\}$. Obviously, each point $x \in \R^n$ can be uniquely
decomposed as $x=y + tu$ where $y \in H$ and $t \in \R$. The Steiner
symmetral of a set $K$ with respect to $u$ is defined to be
\[
S_u(K) = \left\{(y,t)\,:\, K \cap (y+\R u) \ne \emptyset,\quad |t|\le
\frac{|K \cap (y + \R u)|}{2}\right\}.
\]
The Steiner symmetrization has several important properties. For one,
it reduces the surface area while preserving volume. Clearly, this
process makes the set more ``round'' in some sense, so one would expect
that applying mutliple Steiner symmetrizations is a process that
converges to the Euclidean ball - the only fixed point of this
operation. It was shown by Gross \cite{Gross} that for each convex
set there exists a sequence of symmetrizations that converges in the
Hausdorff metric to a ball with the same volume. This result was
improved by Mani-Levitska \cite{Mani} where it was shown that a
random sequence of Steiner symmetrizations applied to a convex set,
converges almost surely to a ball. However, these proofs do not
provide results regarding the rate of convergence. The first estimate
of the rate is due to Hadwiger \cite{Hadwiger}, who showed that
$\left(c\frac{\sqrt{n}}{\varepsilon^2}\right)^n$ symmetrizations are
enough to transform a convex set to a new set with Hausdorff distance
at most $\varepsilon$ from the Euclidean ball. Later, Bourgain,
Lindenstrauss and Milman \cite{BLM} proved an isomorphic result,
stating that in order to reach some fixed distance from the Euclidean
Ball, roughly $n\log n$ symmetrizations suffice. In recent years this
bound was reduced to $3n$ by Klartag and Milman
\cite{KM_Isomorph-Stein}.
Klartag \cite{K_Isomet-Mink+Stein} also improved the isometric result
of Hadwiger, showing that the rate of convergene is almost
exponential. 
More precisely:
\begin{thm}[Klartag]\label{thm-Klartag-Convergence}
Let $K\in\kn$ be a convex body with $|K|=|D_n|$, and let $\varepsilon
\in (0,\frac{1}{2})$. There exist $Cn^4(\log\varepsilon)^2$ Steiner
symmetrization transforming $K$ into a body $K'$ satisfying
\[
(1-\varepsilon)D_n \subset K \subset (1+\varepsilon)D_n.
\]
\end{thm}

In \cite{K_Isomet-Mink+Stein} Klartag first provided a bound of
$Cn|\log \varepsilon |$ steps on the convergence rate when applying
the {\em Minkowski symmetrization} $M_u K$, a linear operation on the
support function, by means of controlling the decay of the
non-constant spherical harmonics of the support function. The proof
of Theorem \ref{thm-Klartag-Convergence} consists mainly of the bound
for Minkowski symmetrizations, together with the inclusion
$S_u K\subseteq M_u K$. A byproduct of this approximation is that the
bound for Steiner symmetrizations is polynomial in the dimension $n$
rather than linear. It is conjectured that the correct dependence is
indeed linear, as in the case of Minkowski symmetrizations. The goal
of this paper is to provide a direct estimate for the convergence
rate of Steiner symmetrizations in the Nikodym pseudo metric, defined
in Section \ref{Sec-Stability}.
It may be formulated as follows, where $A\Delta B$ is the {\em
symmetric difference} of the sets $A$ and $B$.
\begin{thm}\label{Thm_Main}
Let $K\in \kn$ be a convex body with $|K|=|D_n|$ and let $\varepsilon
\in(0,1)$. There exist $c\frac{n^{13}\log^3 n}{\varepsilon^\gamma}$
Steiner symmetrizations transforming $K$ into a body $K'$ satisfying
\[
\frac{|K' \Delta D_n|}{|D_n|} < \varepsilon,
\]
where $\gamma={4+\frac{2}{\log n}}$.
\end{thm}
Obviously, Theorem \ref{Thm_Main} provides a non optimal bound (for
example, by equivalence of the Hausdorff and Nikodym metrics, one can
derive a better bound from Theorem \ref{thm-Klartag-Convergence}).
However, the polynomial bound presented in this proof is obtained
using a self contained, direct analysis of Steiner symmetrization,
which may lead to similar results in the case of non convex sets,
where there are no estimates analogous to Theorem \ref{thm-Klartag-Convergence}.
The main ingredient of our proof is a quantitative estimate regarding
the change in surface area under a Steiner symmetrization. It is a
well known fact that surface area decreases under a Steiner
symmetrization. However, a quantitative version of this statement was
only recently provided, by Barchiesi, Cagnetti and Fusco \cite{BCF}.
Their statement contains factors which are exponential in the
dimension and have a direct effect on the estimate of the convergence
rate. In Section \ref{Sec-Stability} we provide a slightly different
version with an improved dependence on the dimension. To this end we
require a Poincar\'{e} type inequality for convex domains, which we
obtain in the following section.

\section{Poincar\'{e} type inequalities for convex domains}\label{Sec_Poincare}
We wish to establish a weighted Poincar\'{e} type inequality for
convex domains. We denote by $\rho: K\to \R^+$ the distance to the
boundary of $K$, that is
\[
\rho(x)=\min_{y\in \partial K} \{ |x-y| \}.
\]

Our main result in this section is the following theorem:
\begin{thm}\label{Thm-OurPoincare}
Let $n\ge 2$ and let $K \in \kn$ be such that $rD_n \subseteq K \subseteq
R D_n$. If $b: K \to \R$ is a bounded function with mean zero with
respect to $\rho$ (i.e. $\int_K b\rho=0$), then for every $\lambda\in
(2,\infty)$ one has
\[
\int_K |b| \le C 
\left(\frac{||b||_\infty|K|}{\beta} \right)^{1-\frac{1}{\lambda}}
\cdot \left(n\frac{R}{r}\int_K |\grad b|\rho\right)^{\frac{1}{\lambda}},
\]
where $\beta=\frac{\lambda-2}{\lambda-1}\in (0,1)$.
\end{thm}

We collect a few technical lemmas before proving Theorem \ref{Thm-OurPoincare}.
\begin{lem}\label{Lem_Switch-Vol-Surf.Area}
Let $n\ge 1$ and let $K\in\kn$ be such that $rD_n \subseteq K \subseteq
R D_n$. Then \[
\frac{n}{R} \le
\frac{|\partial K|}{|K|} \le
\frac{n}{r}
.\]
\end{lem}
\begin{proof}
By the definition we have: \[
|\partial K| =
\lim_{t\to 0} \frac{ |K + t D_n| - |K| }{t} \le
\lim_{t\to 0} \frac{ |K + \frac{t}{r} K| - |K| }{t} =
\frac{n |K|}{r}
.\]
Replacing $D_n$ with $K/R$ in the above limit yields the other direction.
\end{proof}

\begin{lem}\label{Lem_Bound-1-over-Rho}
Let $n\ge 1$, $K\in\kn$.
For every $\beta\in (0,1)$ we have \[
I_{\beta} = \int_K \frac{1}{\rho^{1-\beta}} <
\frac{C n^{1-\beta}|K|}{\beta r^{1-\beta}},\]
where $r$ is the inner radius of $K$, and $C$ is some positive constant.
\end{lem}

\begin{proof}
First, recall the {\em Beta function} defined for positive $x$ and $y$ by \[
B(x,y) = \int_0^1 t^{x-1}(1-t)^{y-1}dt.\]
The function $\rho$ is bounded (from above and below) by the $K$-distance-to-the-boundary function $\rho_K(x)=\min_{y\in \partial K}
\{ ||x-y||_K \}$ $= 1 - ||x||_K$, whose corresponding integral is easily
estimated. Indeed, if $rD_n \subseteq K \subseteq RD_n$ then
\begin{equation*}
\frac{1}{R}|x| \le ||x||_K \le \frac{1}{r}|x|.
\end{equation*}
In particular we get a lower bound on $\rho$:
\begin{equation}\label{eq-rho_bound}
\rho(x) = \min_{y\in \partial K}|x-y| \ge r \min_{y\in\partial K}||x-y||_K = r \rho_K(x) .
\end{equation}
By Fubini's theorem:
\begin{eqnarray*}
\int_K \frac{1}{\rho_K(x)^{1-\beta}}
&=&
\int_0^\infty \left|\left\{x \in K : \rho_K^{-(1-\beta)}(x) > t\right\}\right|dt \\ \\
&=&
\int_0^1 |K| dt + \int_1^\infty \left|\left\{x \in K : \rho_K^{-(1-\beta)}(x) > t\right\}\right|dt \\ \\
&=&
|K|\left(1 + \int_1^\infty\left(1 - (1-t^{-\frac{1}{1-\beta}})^n \right)  dt\right) \\ \\
&=&
|K|\left(1 + \int_0^1 (1-s^n)(1-s)^{\beta-2}(1-\beta)ds\right) \\ \\
&=&
n|K|\int_0^1(1-s)^{\beta-1}s^{n-1}ds = n|K|B(n,\beta) \\ \\
&=&
n|K|\frac{n+\beta}{\beta}B(n, 1+\beta) \\ \\
&=&
n|K|\left(\frac{n+\beta}{\beta}\right)\frac{\Gamma(n)\Gamma(1+\beta)}{\Gamma(1+n+\beta)} \\ \\
&<&\frac{2 n^2 |K| \Gamma(n)}{\beta \Gamma(1+n+\beta)}
<
\frac{C n^{1-\beta} |K|}{\beta},
\end{eqnarray*}
for some $C > 0$, since $\Gamma(n)n^{1+\beta}< C_1\Gamma(1+n+\beta)$.
Therefore, by \eqref{eq-rho_bound} we conclude that
\begin{eqnarray*}
I_{\beta} &=& \int_K \frac{1}{\rho^{1-\beta}} \leq \frac{1}{r^{1-\beta}}\int_K \frac{1}{\rho_K(x)^{1-\beta}}
< \frac{C n^{1-\beta}|K|}{\beta r^{1-\beta}}.
\end{eqnarray*}
\end{proof}

The last tool we require is the following weighted Poincar\'{e} type
inequality, due to Chua and Wheeden (in fact, in \cite{ChW} they
prove a more general result).
\begin{thm}[Chua, Wheeden]\label{thm-PoincareChua}
Let $K \in \kn$ and let $f$ be a Lipschitz function. If
\[
\int_K f\rho = 0,
\]
then
\[
\int_K |f|\rho  \leq C \text{diam}(K) \int_K |\grad f|\rho,
\]
where $C>0$ is some universal constant.
\end{thm}

We turn now to prove the main result of this section.
\begin{proof}[{\bf Proof of Theorem \ref{Thm-OurPoincare}.}]
Let $\lambda > 2$. By the H\"{o}lder inequality we have
\begin{equation} \label{eq-Reverse_Holder}
\int_K|b|^{\frac{1}{\lambda}}\le
\left(\int_K |b|\rho\right)^{\frac{1}{\lambda}} \cdot
\left(\int_K \rho^{\frac{1}{1-\lambda}}\right)^{1-\frac{1}{\lambda}}
%
\end{equation}
We write $\lambda = \frac{2-\beta}{1-\beta}$ for $\beta \in(0,1)$, so
that $(\lambda-1)(1-\beta)=1$. By Lemma \ref{Lem_Bound-1-over-Rho},
\begin{eqnarray*}
\left(\int_K \rho^{\frac{1}{1-\lambda}}\right)^{1-\frac{1}{\lambda}} &=& \left(\int_K \frac{1}{\rho^{1-\beta}}\right)^{1-\frac{1}{\lambda}} <
\left(\frac{C n^{1-\beta}|K|}{\beta r^{1-\beta}}\right)^{1-\frac{1}{\lambda}} \\ \\
&=& \left(\frac{n}{r}\right)^{\frac{1}{\lambda}}\cdot
\left(\frac{C|K|}{\beta}
\right)^{1-\frac{1}{\lambda}}
\end{eqnarray*}
Combining the two estimates we get
\[
\int_K |b| \le
||b||_\infty^{1-\frac{1}{\lambda}}\int_K|b|^{\frac{1}{\lambda}}\le
\left(\frac{C||b||_\infty|K|}{\beta}  \right)^{1-\frac{1}{\lambda}}
\left(\frac{n}{r}\int_K |b|\rho\right)^{\frac{1}{\lambda}}.
\]
Since $\int_K b\rho = 0$, and $diam(K) \le 2R$, we may apply Theorem \ref{thm-PoincareChua} to obtain:
\[
\int_K |b| \le
\left(\frac{C||b||_\infty|K|}{\beta}  \right)^{1-\frac{1}{\lambda}}
\left(\frac{nR}{r}\int_K |\grad b|\rho\right)^{\frac{1}{\lambda}}.
\]
\end{proof}

\section{Stability for the Steiner symmetrization}\label{Sec-Stability}
Let $K \in \kn$ and $u \in S^{n-1}$. It is well known that the surface
area $|\partial K|$ decreases under a Steiner symmetrization, but until
recently this phenomenon was not quantified. Barchiesi, Cagnetti and
Fusco showed in \cite{BCF} that for a convex body $K$ satisfying
 $rD_n \subseteq K \subseteq RD_n$, the following holds
\begin{equation} \label{eq-FuscoStability}
A(K, S_u K) \leq n4^{n+1}\left(\frac{R}{r}\right)^{2n}\sqrt{\delta_u(K)},
\end{equation}
where the {\em surface area deficit} $\delta_u$ is defined by
\[
\delta_u(K) = 1 - \frac{|\partial S_u K|}{|\partial K|},
\]
and the {\em Nikodym pseudo metric} $A$ is defined by
\[
A(K, T) = \inf_{x_0 \in \R^n}\frac{|(rK) \Delta (x_0 + T))|}{|T|},
\]
for  $r^n=\frac{|T|}{|K|}$.
In this section we show that the dependence on the dimension and the
quantity $\frac{R}{r}$ in (\ref{eq-FuscoStability}) can be reduced to
polynomial at the cost of slightly worsening the exponent of
$\delta_u(K)$ (i.e. decreasing it below $1/2$). More precisely:

\begin{thm}\label{thm-SteinerStability}
Let $n\ge 2$ and let $K \in \kn$ such that $rD_n \subseteq K \subseteq
R D_n$. Then for every $\lambda\in (2,\infty)$ and $u\in S^{n-1}$ we have:
\[ A(K,S_u K)\le  C
\left(\frac{1}{\beta}  \right)^{1-\frac{1}{\lambda}}
\left( n\frac{R}{r} \right)^{1+\frac{1}{\lambda}}
\delta_u(K)^{\frac{1}{2\lambda}},
\]
where $\beta=\frac{\lambda-2}{\lambda-1}\in (0,1)$, and $C$ is some constant.
\end{thm}

The proof of Theorem \ref{thm-SteinerStability} follows the methods of
\cite{BCF}, combined with Theorem \ref{Thm-OurPoincare}. Denote the
orthogonal projection of $K$ to $u^\perp$ by $P = Proj_{u^\perp}(K)$.
For each $x \in P$, we consider the ``fiber above $x$ in $K$'', namely
$K \cap (x+\R u)$. We denote its length by $L(x) = |K \cap (x+\R u)|$
and its barycenter by $b(x)$.
By Brunn's principle, $L$ is concave. The following lemma gives a
local upper bound for the gradient of a concave function, in terms of
the distance from the boundary of its domain.

\begin{lem}\label{Lem_Bound_grad_L}
Let $P \in \kn$, and denote by $\rho(x)=dist(x, \partial P)$ the
distance to $\partial P$. If $L:P\to\R$ is a concave function with
oscillation $\Delta L:=\sup\{L\}-\inf\{L\}$, then
\[
|\grad L(y)| \leq \frac{\Delta L}{\rho(y)}.
\]
\end{lem}
\begin{proof}
Let $y\in P$, and consider $x = y - \rho(y)v \in P$, where
$v = \frac{\grad L(y)}{|\grad L(y)|}$. Then
\[
L(y) - L(x) =
\int_0^{\rho(y)} \frac{\partial L}{\partial v}(x + tv)dt \ge
\int_0^{\rho(y)}\frac{\partial L}{\partial v}(y)dt =
\rho(y)|\grad L(y)|,
\]
by concavity. Therefore $|\grad L(y)|\le \frac{\Delta L}{\rho(y)}$, as required.
\end{proof}

\begin{proof}[{\bf Proof of Theorem \ref{thm-SteinerStability}.}]
As before, denote $P = Proj_{u^\perp}(K)= Proj_{u^\perp}(S_u K)$ and
$\rho(x) = dist(x, \partial P)$. The fiber $K \cap (x+\R u)$ has endpoints with heights $b\pm \frac{L}{2}$, thus:
\begin{eqnarray*}
|\partial K| - |\partial S_u K| &=&
\int_P \left(\sqrt{1+\left|\grad b + \frac{\grad L}{2}\right|^2} + \sqrt{1+\left|\grad b - \frac{\grad L}{2}\right|^2} - 2\sqrt{1 + \left|\frac{\grad L}{2}\right|^2} \right) \\ \\
&=& \int_P \frac{N}{D}\ge
\left(\int_P N^\frac{1}{2}\right)^2
\left(\int_P D\right)^{-1},
\end{eqnarray*}
where
\begin{eqnarray*}
N &=&
\frac{1}{2}\left(\sqrt{1+\left|\grad b + \frac{\grad L}{2}\right|^2} + \sqrt{1+\left|\grad b - \frac{\grad L}{2}\right|^2}\right)^2 -
\frac{1}{2}\left(2\sqrt{1 + \left|\frac{\grad L}{2}\right|^2} \right)^2
\\ \\ &=&
\sqrt{\left(1 + \frac{1}{4}|\grad L|^2 + |\grad b|^2 \right)^2 - \iprod{\grad L}{\grad b}^2} - \left(1 + \frac{1}{4}|\grad L|^2 - |\grad b|^2\right)
\\ \\ &=&
\frac{ 4\left(1+\frac{1}{4}|\grad L|^2 \right)|\grad b|^2 - \iprod{\grad b}{\grad L}^2 } {\sqrt{\left(1 + \frac{1}{4}|\grad L|^2 + |\grad b|^2 \right)^2 - \iprod{\grad L}{\grad b}^2} + \left(1 + \frac{1}{4}|\grad L|^2 - |\grad b|^2\right) } \\ \\
&\geq& \frac{4|\grad b|^2}{\sqrt{\left(1 + \frac{1}{4}|\grad L|^2 + |\grad b|^2 \right)^2 - \iprod{\grad L}{\grad b}^2} + \left(1 + \frac{1}{4}|\grad L|^2 - |\grad b|^2\right)}
\equiv \frac{4|\grad b|^2}{Q},
\end{eqnarray*}
and
\[
D = \frac{1}{2}\left(
\sqrt{1+\left|\grad b + \frac{1}{2}\grad L\right|^2} + \sqrt{1+\left|\grad b - \frac{1}{2}\grad L\right|^2}
\right)
 + \sqrt{1 + \frac{1}{4}|\grad L|^2},
\]
so that $\int_P D = \frac{|\partial K| + |\partial S_u K|}{2} \le
|\partial K|$. Thus:
\begin{equation}\label{eq-surface_deficit_est}
\delta_u(K) = \frac{|\partial K| - |\partial S_u K|}{|\partial K|}\ge
\left(\frac{1}{|\partial K|} \int_P \sqrt{N}\right)^2.
\end{equation}
Next, we bound $N$ from below. Since $\sqrt{a^2-x^2} \leq a-\frac{x^2}{2a}$ we have

\begin{eqnarray*}
Q&=&\sqrt{\left(1 + \frac{1}{4}|\grad L|^2 + |\grad b|^2 \right)^2 - \iprod{\grad L}{\grad b}^2} + \left(1 + \frac{1}{4}|\grad L|^2 - |\grad b|^2\right) \\ \\
&\leq & 2 + \frac{1}{2}|\grad L|^2 - \frac{\iprod{\grad b}{\grad L}^2}{2 + \frac{1}{2}|\grad L|^2 + 2|\grad b|^2} \leq 2 + \frac{1}{2}|\grad L|^2 \\ \\
&\leq& 2\frac{R^2}{\rho^2} + \frac{1}{2}|\grad L|^2 \leq \frac{4R^2}{\rho^2},
\end{eqnarray*}
where the last inequality is due to Lemma \ref{Lem_Bound_grad_L} (here $\Delta L\le 2 R$). Therefore
\[
\sqrt{N} \ge \frac{|\grad b|\rho}{R}.
\]
Plugging this back to \eqref{eq-surface_deficit_est}, we may bound the surface area deficit:
\begin{equation}\label{Eq_Surface-Deficit}
\sqrt{\delta_u(K)} \ge
\frac{1}{R|\partial K|} \int_P |\grad b| \rho.
\end{equation}

In order to bound the Nikodym pseudo metric by the integral of the
barycenter, we first note that since $K\subseteq R D_n$, we have
$|b|\le R$. Moreover, $K$ may be shifted parallel to $u$, so
without loss of generality we may assume $\int_P b\rho=0$. This shift
cannot exceed $R$, thus the new barycenter is bounded by $2R$. We have:
\begin{equation}\label{Eq_Volume-Deficit}
A(K, S_u K) \le \frac{|K \Delta S_u K|}{|K|} \le \frac{1}{|K|}\int_P |b|,
\end{equation}
where the second inequality is due to the fact that
$| [-\frac{L}{2},\frac{L}{2}] \Delta
[b-\frac{L}{2},b+\frac{L}{2}] |\le |b|$. Since we assumed that
$\int_P b\rho=0$, we may apply Theorem \ref{Thm-OurPoincare} to get
(recall $||b||_\infty\le 2R$):
\begin{eqnarray*}
A(K, S_u K) &\le&
\frac{1}{|K|}
\left(\frac{C||b||_\infty|P|}{\beta}  \right)^{1-\frac{1}{\lambda}}
\left(\frac{nR}{r}\int_P |\grad b|\rho\right)^{\frac{1}{\lambda}}
\\ \\ &\le&
\frac{1}{|K|}
\left(\frac{2CR|P|}{\beta}  \right)^{1-\frac{1}{\lambda}}
\left(\frac{nR^2|\partial K|}{r} \sqrt{\delta_u(K)}\right)^{\frac{1}{\lambda}}
\\ \\ &=&
\frac{|\partial K|^{\frac{1}{\lambda}}}{|K|}
\left(\frac{2C|P|}{\beta}  \right)^{1-\frac{1}{\lambda}}
\frac{n^\frac{1}{\lambda}R^{1+\frac{1}{\lambda}}}{r^\frac{1}{\lambda}}
\delta_u(K)^{\frac{1}{2\lambda}}
\\ \\ &\le&
\frac{|\partial K|}{|K|}
\left(\frac{C}{\beta}  \right)^{1-\frac{1}{\lambda}}
\frac{n^\frac{1}{\lambda}R^{1+\frac{1}{\lambda}}}{r^\frac{1}{\lambda}}
\delta_u(K)^{\frac{1}{2\lambda}}
\\ \\ &\le&
\left(\frac{C}{\beta}  \right)^{1-\frac{1}{\lambda}}
\left( n\frac{R}{r} \right)^{1+\frac{1}{\lambda}}
\delta_u(K)^{\frac{1}{2\lambda}}.
\end{eqnarray*}
The last two inequalities hold since $P$ is a $n-1$ dimensional set
contained in $S_u K$, thus $2|P| \le |\partial S_u K| \le
|\partial K|$. Moreover, by Lemma \ref{Lem_Switch-Vol-Surf.Area}, $\frac{|\partial K|}{|K|}\le \frac{n}{r}$.
\end{proof}

\section{Rate of convergence}
In this section we prove Theorem \ref{Thm_Main}. The proof is based
on the following idea. Assume that $|K|=|D_n|$. Due to Theorem
\ref{thm-SteinerStability}, as long as one can find a direction $u$
for which $A(K, R_uK)$ is not very small, there exists a Steiner
symmetrization which reduces the surface area of $K$ by a factor.
Since the surface area cannot drop below $n|D_n|$ (isoperimetric
inequality), the number of such operations is bounded. Next, one has
to show that if $A(K, R_uK)$ is small in every direction, then so is
$A(K,D_n)$. Let us formulate this last statemnt precisely before
proving the main theorem.

\begin{lem}\label{lem-CloseToReflections}
Let $K\subset \R^n$ be a compact star shaped body and let $\varepsilon
>0$. Denote by $R_u$ the reflection with respect to $u^\perp$. If
$A(K, R_uK)<\varepsilon$ for all
$u\in S^{n-1}$, then $A(K, D_n) < 4n\varepsilon$.
\end{lem}
\begin{proof}
First note that $A(K, R_{u_m}\dots R_{u_1}K) < m\varepsilon$ for any
$m \le n$. The proof goes by induction, where the case $m=1$ is
assumed to hold. For $m \ge 2$ one has
\begin{eqnarray*}
A(K, R_{u_m}\ldots R_{u_1}K) &=& A(R_{u_m}K, R_{u_{m-1}}\ldots R_{u_1}K) \\
 &\le& A(R_{u_m}K, K) + A(K, R_{u_{m-1}}\ldots R_{u_1}K)  \\
&<& \varepsilon + (m-1)\varepsilon = m\varepsilon.
\end{eqnarray*}
Every isometry $u\in O(n)$ is generated by at most $n$ reflections,
thus $A(K, uK) < n\varepsilon$. This may be written as follows, in
terms of the radial function $\rho$ of $K$:
\begin{equation}\label{Eq_SymmDiffIsometry}
A(K, uK) = \frac{|K \Delta uK|}{|K|} = \frac{|D_n|}{|K|}
\int_{S^{n-1}} |\rho(x)^n - \rho(ux)^n| d\sigma(x) < n\varepsilon,
\end{equation}
where $\sigma$ is the normalized Haar measure on the sphere. Without
loss of generality, assume from now on that $|K|=|D_n|$. Note that if
$u$ is selected at random with respect to the Haar measure on
$SO(n)$, then for every $x\in S^{n-1}$, the point $ux$ is distributed
uniformly on $S^{n-1}$. Thus averaging (\ref{Eq_SymmDiffIsometry})
over $u\in SO(n)$ yields:
\begin{equation}\label{Eq_SymmDiffBall}
\int_{S^{n-1}}
\int_{S^{n-1}} |\rho(x)^n - \rho(y)^n| d\sigma(x) d\sigma(y)
< n\varepsilon.
\end{equation}
Consider the sets $A=\{x \in S^{n-1} : \rho(x) \ge 1\}$ and
$B=\{x \in S^{n-1} : \rho(x) \le 1\}$. Since $|K \setminus D_n| =
|D_n \setminus K|$, we have
\begin{eqnarray*}
\frac{1}{2}A(K, D_n) &=& \frac{1}{2}
\int_{S^{n-1}} |\rho(x)^n - 1| d\sigma(x) =
\int_A         |\rho(x)^n - 1| d\sigma(x)\\ \\ &=&
\frac{1}{\sigma(B)} \int_B \int_A
|\rho(x)^n - 1| d\sigma(x)d\sigma(y)\\ \\&\le&
\frac{1}{\sigma(B)} \int_B \int_A
|\rho(x)^n - \rho(y)^n| d\sigma(x) d\sigma(y)\\ \\&\le&
\frac{1}{\sigma(B)} \int_{S^{n-1}} \int_{S^{n-1}}
|\rho(x)^n-\rho(y)^n|d\sigma(x)d\sigma(y).
\end{eqnarray*}
This implies $A(K,D_n) < \frac{2n\varepsilon}{\sigma(B)}$ by
\eqref{Eq_SymmDiffBall}, and similarly one has $A(K,D_n) <
\frac{2n\varepsilon}{\sigma(A)}$, so combining the two we get
\[
A(K, D_n) < 4n\varepsilon. 
\]
\end{proof}

\begin{proof}[{\bf Proof of Theorem \ref{Thm_Main}.}]
Assume without loss of generality that $|K| = |D_n|$.
Apply $n$ Steiner symmetrizations to $K$ with respect to some
orthogonal basis to obtain a new convex body $K_0$ which is
unconditional, and in particular centrally symmetric. By John's
theorem, there exists an ellipsoid $\mathcal{E}$ such that
\[
\mathcal{E} \subset K_0 \subset  \sqrt{n}\mathcal{E}.
\]
There exist $n$ Steiner symmetrizations which transform $\mathcal{E}$
to an Euclidean ball (see \cite{KM_Isomorph-Stein}, Lemma
2.6). Applying these symmetrizations to $K_0$, we obtain a body $K_1$
satisfying
\[
r_1D_n \subset K_1 \subset \sqrt{n}r_1D_n,
\]
for some $r_1 > 0$. Thus the inner and outer radii of $K_1$ satisfy
$\frac{R}{r} \le \sqrt{n}$. Note that $|D_n| = |K_1| \le |\sqrt{n}
rD_n|$ which implies that $\frac{1}{r}\le \sqrt{n}$. Hence, by Lemma
\ref{Lem_Switch-Vol-Surf.Area} 
\[
|\partial K_1| \leq \frac{n}{r}|K_1| \le n^{3/2}|D_n|.
\]
Fix $\varepsilon_0 > 0$. If there exists $u_1 \in S^{n-1}$
with $A(K_1, S_{u_1}K_1) > \varepsilon_0$, denote $K_2 = S_{u_1} K_1$.
By Theorem \ref{thm-SteinerStability}, combined with the bound $\frac{R}{r} \le \sqrt{n}$, we get
\begin{eqnarray*} \label{eq-SurfAreaReduction}
n|D_n| = |\partial D_n| \le |\partial K_2| &\le&
|\partial K_1|
\left(1 - \frac{A(K_1, S_{u_1} K_1)^{2\lambda}}{n^{3(\lambda+1)} }\left(\frac{\beta}{C}\right)^{2(\lambda-1)}
 \right)
\\ \\ &\le &
|\partial K_1|
\left( 1 -
\frac{\varepsilon_0^{2\lambda}}{n^{3(\lambda+1)} } \left(\frac{\beta}{C}\right)^{2(\lambda-1)} \right).
\end{eqnarray*}
If there exists $u_2\in S^{n-1}$ with $A(K_2, S_{u_2}K_2) >
\varepsilon_0$, denote by $K_3 = S_{u_2}K_2$. Continue this process
for $m$ steps. Then $K_{m+1}$ satisfies
\begin{eqnarray*}
n|D_n| \le |\partial K_{m+1}| &\le& |\partial K_1|\left(
1 - \left(\frac{\beta}{C}\right)^{2(\lambda-1)}
\frac{\varepsilon_0^{2\lambda}}{n^{3(\lambda+1)}} \right)^m
\\ \\ &\le &
n^{3/2}|D_n|\left( 1 - \left(\frac{\beta}{C}\right)^{2(\lambda-1)}
\frac{\varepsilon_0^{2\lambda}}{n^{3(\lambda+1)} } \right)^m.
\end{eqnarray*}
Thus, 
\[
0 \leq \frac{1}{2}\log n + m \log \left( 1 - \left(\frac{\beta}{C}\right)^{2(\lambda-1)}
\frac{\varepsilon_0^{2\lambda}}{n^{3(\lambda+1)} } \right).
\]
Set $\lambda = 2 + \frac{1}{\log n}$, so $\beta=\frac{1}{1 +\log n}$.
Hence, the number of such steps is bounded by 
\begin{equation}\label{Eq_mBound}
m \le \left(C(1 +\log n)\right)^{2 + \frac{2}{\log n}}\log n
\left(
\frac{n^{9 + \frac{3}{\log n}}}{\varepsilon_0^{4 + \frac{2}{\log n}}}
\right) <
c \left(
\frac{ n^9 \log^3n}{\varepsilon_0^{4 + \frac{2}{\log n}}}
\right),
\end{equation}
for some $c>0$. The resulting body $K'$ thus satisfies $A(K', S_u K') < \varepsilon_0$ for all $u \in S^{n-1}$, which in turn implies that $A(K', R_u K') < 2\varepsilon_0$ for all $u \in S^{n-1}$, where $R_u K_m$ is the reflection of $K_m$ with respect to $u^\perp$. By Lemma \ref{lem-CloseToReflections} we conclude that 
\[
A(K_m, D_n) < 8n\varepsilon_0.
\]
Let $\varepsilon>0$. Plugging $\varepsilon_0 = \varepsilon/(8n)$ into
\eqref{Eq_mBound} completes the proof.
\end{proof}

\begin{rem}{\rm
The dependence in the dimension $n$ in Theorem \ref{Thm_Main} is
clearly not optimal (as mentioned before, the sharp bound is believed
to be linear). For example, the bound for the ratio $R/r$ may be
reduced to a constant, rather than $\sqrt{n}$, which results in
decreasing the power $13$ to $10$. This may be done by one of the
isomrphic results mentioned in the introduction.
}\end{rem}

\noindent Dan Itzhak Florentin, danflorentin@gmail.com\\
\noindent Department of Mathematics, Weizmann Institute of Science, Rehovot, Israel.\\

\noindent Alexander Segal, segalale@gmail.com\\
\noindent School of Mathematical Science, Tel Aviv University, Tel Aviv, Israel.
\end{document}